\documentclass[11pt]{amsproc}
 \usepackage[margin=1in]{geometry}
\usepackage{setspace,fullpage}
\usepackage{graphicx}
\usepackage[nice]{nicefrac}
\usepackage{amssymb}
\usepackage{multirow}
\usepackage{array}

\usepackage{amsmath,amsthm,amscd,amssymb, mathrsfs}

\usepackage{latexsym}

\numberwithin{equation}{section}

\theoremstyle{plain}
\newtheorem{theorem}{Theorem}[section]
\newtheorem{lemma}[theorem]{Lemma}
\newtheorem{corollary}[theorem]{Corollary}

\newtheorem{conjecture}[theorem]{Conjecture}

\theoremstyle{definition}
\newtheorem{definition}[theorem]{Definition}

\theoremstyle{remark}
\newtheorem{remark}[theorem]{Remark}

\newtheorem{case[theorem]}{Case}

\title[\parbox{14cm}{\centering{Improved extension Theorems \hspace{1in}}} \quad]{Conjecture and improved extension theorems for paraboloids in the finite field setting}
\author{ Doowon Koh }

\address{Department of Mathematics\\
Chungbuk National University \\
Cheongju, Chungbuk 28644 Korea}
\email{koh131@chungbuk.ac.kr}

\thanks{
This research was supported by Basic Science Research Program through the National Research Foundation of Korea(NRF) funded by the Ministry of Education, Science and Technology(NRF-2015R1A1A1A05001374)
}


\begin{document}

\begin{abstract} We study the extension estimates for paraboloids in $d$-dimensional vector spaces over finite fields $\mathbb F_q$ with $q$ elements.
We use the connection between $L^2$ based restriction estimates and $L^p\to L^r$ extension estimates for paraboloids. As a consequence, we improve the $L^2\to L^r$ extension results obtained by A. Lewko
and M. Lewko \cite{LL10}  in even dimensions $d\ge 6$ and odd dimensions $d=4\ell+3$ for $\ell \in \mathbb N.$ Our results extend the consequences for 3-D paraboloids due to M. Lewko \cite{LL13} to higher dimensions. We also clarifies conjectures on finite field extension problems for paraboloids.

\end{abstract}

\maketitle
\section{Introduction} Let $V\subset \mathbb R^d$ be a hypersurface which is endowed with a surface measure $d\sigma.$
In the Euclidean setting, the extension problem is to determine the exponents $1\le p, r\le \infty$ such that the following inequality holds:
$$ \|(fd\sigma)^\vee\|_{L^{r}(\mathbb R^d)} \le C \|f\|_{L^{p}(V, d\sigma)},$$
where the constant $C>0$ is independent of functions $f\in L^p(V, d\sigma).$ By duality, this extension estimate is same as the restriction estimate 
$$\|\widehat{g}\|_{L^{p'}(V, d\sigma)} \le C\|g\|_{L^{r'}(\mathbb R^d)}.$$
Here, $p'$ and $r'$ denote the H\"{o}lder conjugates of $p$ and $r$, respectively (i.e. $1/p + 1/p'=1$).
Therefore, the extension problem is also called the restriction problem.
In 1967, E.M. Stein \cite{St78} introduced the restriction problem. This problem had been completely solved for the parabola and the circle in two dimensions, and the cones in three and four dimensions (see \cite{Zy74, Ba85, Wo01}).
However, it is still open in other cases although  improved results have been obtained by harmonic analysts.
We refer readers to \cite{Gu15, St93,Ta03, Ta04} for further information and recent developments on the restriction problem in the Euclidean setting.\\

In 2002, Mockenhaupt and Tao \cite{MT04} initially posed and studied the extension problem for various varieties in $d$-dimensional vector spaces over finite fields.
In order to formulate a finite field analogue of the extension problem,  the real set is replaced by finite fields. We begin by reviewing the definition of the finite field extension problem.
We denote by $\mathbb F_q$ a finite field with $q$ elements.
Throughout this paper, we shall assume that $q$ is a power of odd prime. Let $\mathbb F_q^d$ be a $d$-dimensional vector space over the finite field $\mathbb F_q.$
We endow the vector space $\mathbb F_q^d$ with the counting measure $dm.$ We write $(\mathbb F_q^d, dm)$ to stress that the vector space $\mathbb F_q^d$ is endowed with the counting measure $dm.$
Since the vector space $\mathbb F_q^d$ is isomorphic to its dual space as an abstract group, we identify the space $\mathbb F_q^d$ with its dual space.
However,  a normalized counting measure $d\xi$ is endowed with its dual space which will be denoted by  $(\mathbb F_q^d, d\xi).$ We always use the variable $m$ for an element of the vector space $(\mathbb F_q^d, dm)$.
On the other hand, the variable $\xi$ will be an element of the dual space $(\mathbb F_q^d, d\xi).$ For example, we simply write $m\in \mathbb F_q^d$ and $\xi \in \mathbb F_q^d$ for $m\in (\mathbb F_q^d, dx)$ and $\xi \in (\mathbb F_q^d, d\xi)$, respectively. For a complex valued function $g: (\mathbb F_q^d, dm)\to \mathbb C$, the Fourier transform $\widehat{g}$  on $(\mathbb F_q^d, d\xi)$ is defined by
$$ \widehat{g}(\xi)=\int_{\mathbb F_q^d} g(m) \chi(-m\cdot \xi)\,dm = \sum_{m\in \mathbb F_q^d} g(m)\chi(-m\cdot \xi)$$
where $\chi$ denotes a nontrivial additive character of $\mathbb F_q$ and the dot product is defined by $m\cdot \xi=m_1\xi_1 + \cdots + m_d \xi_d$ for $m=(m_1,\ldots,m_d),\, \xi=(\xi_1,\ldots, \xi_d)\in \mathbb F_q^d.$
For a complex valued function $f:(\mathbb F_q^d, d\xi) \to \mathbb C$,  the inverse Fourier transform $f^\vee$  on $(\mathbb F_q^d, dm)$ is given by
$$ f^\vee(m)=\int_{\mathbb F_q^d} f(\xi) \chi(\xi\cdot m) \,d\xi = \frac{1}{q^d} \sum_{\xi\in \mathbb F_q^d} f(\xi) \chi(\xi\cdot m).$$
Using the orthogonality relation of the nontrivial character $\chi$ of $\mathbb F_q$, we obtain the Plancherel theorem:
$$\|\widehat{g}\|_{L^2(\mathbb F_q^d, d\xi)} = \|g\|_{L^2(\mathbb F_q^d, dm)} \quad \mbox{or}\quad \|f\|_{L^2(\mathbb F_q^d, d\xi)}=\|f^\vee\|_{L^2(\mathbb F_q^d, dm)}.$$
Namely, the Plancherel theorem yields the following equation
$$ \frac{1}{q^d} \sum_{\xi\in \mathbb F_q^d} |\widehat{g}(\xi)|^2 =\sum_{m\in \mathbb F_q^d} |g(m)|^2 \quad \mbox{or}\quad \frac{1}{q^d} \sum_{\xi\in \mathbb F_q^d} |f(\xi)|^2 =\sum_{m\in \mathbb F_q^d} |f^\vee(m)|^2. $$
Notice by the Plancherel theorem that if  $G, F\subset \mathbb F_q^d$, then we have
$$\frac{1}{q^d} \sum_{\xi\in \mathbb F_q^d} |\widehat{G}(\xi)|^2 =|G| \quad \mbox{and} \quad \sum_{m\in \mathbb F_q^d} |F^\vee(m)|^2 =\frac{|F|}{q^d},$$
where $|E|$ denotes the cardinality of a set $E\subset \mathbb F_q^d.$
Here, and throughout this paper, we shall identify the set $E\subset \mathbb F_q^d$ with the indicator function $1_E$ on the set $E.$
Namely, we shall write $\widehat{E}$ for $\widehat{1_E}$, which allows us to use a simple notation.
Given functions $g_1, g_2: (\mathbb F_q^d, dm) \to \mathbb C,$ the convolution function $g_1\ast g_2$ on $(\mathbb F_q^d, dm)$ is defined by
$$ g_1\ast g_2(n) = \int_{\mathbb F_q^d} g_1(n-m) g_2(m)\,dm = \sum_{m\in \mathbb F_q^d} g_1(n-m) g_2(m).$$
On the other hand, if $f_1, f_2: (\mathbb F_q^d, d\xi) \to \mathbb C,$ then the convolution function $f_1\ast f_2$ on $(\mathbb F_q^d, d\xi)$ is given by
$$ f_1\ast f_2(\eta)=\int_{\mathbb F_q^d} f_1(\eta-\xi) f_2(\xi)\,d\xi = \frac{1}{q^d} \sum_{\xi\in \mathbb F_q^d} f_1(\eta-\xi) f_2(\xi).$$
Then it is not hard to see that
$$ \widehat{g_1\ast g_2} = \widehat{g_1} \widehat{g_2} \quad \mbox{and}\quad (f_1\ast f_2)^\vee = f_1^\vee f_2^\vee.$$

Given an algebraic variety $V\subset (\mathbb F_q^d, d\xi)$, we endow $V$ with the normalized surface measure $d\sigma$ which is defined by the relation
$$ \int_V f(\xi)\,d\sigma(\xi) =\frac{1}{|V|} \sum_{\xi \in  V} f(\xi).$$
Notice that $d\sigma(\xi)=\frac{q^d}{|V|}\, 1_V(\xi)\, d\xi$ and  we have
$$ (fd\sigma)^\vee(m)=\int_V f(\xi) \chi(m\cdot \xi)\, d\sigma(\xi) =\frac{1}{|V|} \sum_{\xi\in V} f(\xi) \chi(m\cdot \xi).$$
For each $1\le p,r\le \infty$, we define $R^*_V(p\to r)$ as the smallest positive real number such that the following extension estimate holds:
$$ \|(fd\sigma)^\vee\|_{L^{r}(\mathbb F_q^d, dm)} \le R^*_V(p\to r) \,\|f\|_{L^{p}(V, d\sigma)} \quad \mbox{for all functions}~~f:V \to \mathbb C.$$
By duality, $R^*_V(p\to r)$ is also the smallest positive constant such that the following restriction estimate holds:
$$\|\widehat{g}\|_{L^{p'}(V, d\sigma)} \le R^*_V(p\to r) \,\|g\|_{L^{r'}(\mathbb F_q^d, dm)} \quad\mbox{for all functions}~~g:(\mathbb F_q^d, dm) \to \mathbb C.$$
The number $R^*_V(p\to r)$ may depend on $q$, the size of the underlying finite field $\mathbb F_q.$
The main question on the extension problem for $V\subset \mathbb F_q^d$ is to determine $1\le p, r\le \infty$ such that the number $R^*_V(p\to r)$ is independent of $q.$
Throughout this paper, we shall use $X\lesssim Y$ for $X, Y>0$ if there is a constant $C>0$ independent of $q=|\mathbb F_q|$ such that
$ X\le C Y.$ We also write $Y\gtrsim X$ for $X\lesssim Y,$ and $X\sim Y$ means that $X\lesssim Y$ and $Y\lesssim X.$
In addition, we shall use $X\lessapprox Y$ if for every $\varepsilon>0$ there exists $C_{\varepsilon}>0$ such that $X\lesssim C_{\varepsilon} q^{\varepsilon} Y.$
This notation is handy for suppressing powers of $\log{q}.$
Using the notation $\lesssim$, the extension problem for $V$ is to determine $1\le p,r\le \infty$ such that $R^*_V(p\to r)\lesssim 1.$ \\

Since the finite filed extension problem was addressed in 2002 by Mockenhaupt and Tao \cite{MT04}, it  has been studied for several algebraic varieties such as paraboloids, spheres, and cones (see, for example, \cite{LL13, LL10, KS12, IK10, KS13}.) In particular, very interesting results have been recovered for paraboloids.
From now on, we restrict ourselves to the study of the extension problem for the paraboloid $P\subset \mathbb (\mathbb F_q^d, d\xi)$  defined as
\begin{equation}\label{defP} P= \{\xi\in \mathbb F_q^d: \xi_d=\xi_1^2+ \cdots +\xi_{d-1}^2\}.\end{equation}
This paper is written to achieve two main goals.
One is to address  clarified conjectures on the extension problem for paraboloids. The other is to improve the previously known $L^2\to L^r$ extension estimates for paraboloids in higher dimensions.\\

In Section \ref{secII}, we shall introduce neat necessary conditions which we may conjecture as sufficient conditions for $R_P^*(p\to r)\lesssim 1.$
In particular, by Lemma \ref{GeN} in Section \ref{secII} it is natural  to conjecture the following statement on the $L^2\to L^r$ extension problem for paraboloids.

\begin{conjecture}\label{Conj1} Let $P\subset \mathbb F_q^d$ be the paraboloid defined as in \eqref{defP}. Then we have
\begin{enumerate}
\item If $d\ge 2$ is even, then  $ R_P^*(2\to r)\lesssim 1 \iff \frac{2d+4}{d}\le r\le \infty$
\item If $d=4\ell-1$ for $\ell\in \mathbb N$, and $ -1\in \mathbb F_q$ is not a square number, then we have
$$ R_P^*(2\to r)\lesssim 1 \iff \frac{2d+6}{d+1}\le r\le \infty$$
\item  If $d=4\ell+1$ for $\ell \in \mathbb N$, then $ R_P^*(2\to r)\lesssim 1 \iff \frac{2d+2}{d-1}\le r\le \infty$
\item  If $d\ge 3$ is odd, and $-1\in \mathbb F_q$ is a square number, then we have
$$R_P^*(2\to r)\lesssim 1 \iff \frac{2d+2}{d-1}\le r\le \infty.$$
\end{enumerate}
\end{conjecture}

In the conclusions of Conjecture \ref{Conj1},  the statements for $``\Longrightarrow"$ direction follow immediately from Lemma \ref{GeN} in the following section.
Hence,  Conjecture \ref{Conj1} can be reduced to the following critical endpoint estimate,
 because $R^*_P(2\to r_1) \ge R^*_P(2\to r_2)$ for $1\le r_1\le r_2 \le \infty.$

\newpage
\begin{conjecture}\label{Conj2}
Let $P\subset \mathbb F_q^d$ be the paraboloid defined as in \eqref{defP}. Then we have
\begin{enumerate}
\item If $d\ge 2$ is even, then  $ R_P^*\left(2\to\frac{2d+4}{d} \right)\lesssim 1$
\item If $d=4\ell-1$ for $\ell\in \mathbb N$, and $ -1\in \mathbb F_q$ is not a square number, then
$ R_P^*\left(2\to \frac{2d+6}{d+1}\right)\lesssim 1$
\item  If $d=4\ell+1$ for $\ell \in \mathbb N$, then $ R_P^*\left(2\to \frac{2d+2}{d-1}\right) \lesssim 1$
\item  If $d\ge 3$ is odd, and $-1\in \mathbb F_q$ is a square number, then
$R_P^*\left(2\to \frac{2d+2}{d-1}\right)\lesssim 1.$
\end{enumerate}
\end{conjecture}

\subsection{Statement of main results}
By the Stein-Tomas argument, Mockenhaupt and Tao \cite{MT04} already showed that  the statements $(3), (4)$ in Conjecture \ref{Conj2} are true.
In fact, they proved that $R_P^*(2 \to (2d+2)/(d-1)) \lesssim 1 $ for all dimensions $d\ge 2$ without further assumptions.\\

The statements $(1), (2)$  in Conjecture \ref{Conj2} are very interesting in that the conjectured results are better than the Stein-Tomas inequality which is sharp in the Euclidean case.
This is due to number theoretic issue which we can enjoy when we study harmonic analysis in finite fields.
In dimension two, the statement $(1)$ in Conjecture \ref{Conj2} was already proved by Mockenhaupt and Tao \cite{MT04}, but it is open in higher even dimensions.
For higher even dimensions $d\ge 4,$ Iosevich and Koh \cite{IK09} proved that $R^*_P(2\to 2d^2/(d^2-2d+2))\lessapprox 1$ which improves the Stein-Tomas inequality due to Mockenhaupt and Tao.
This result was obtained by using a connection between $L^p\to L^4$ extension results and $L^2\to L^r$ extension estimates.
In \cite {LL10}, A. Lewko and M. Lewko improved the result of Iosevich and Koh by recovering the endpoint.
They adapted the bilinear approach to derive the improved result, $R^*_P(2\to 2d^2/(d^2-2d+2))\lesssim 1.$
In this paper, we shall obtain further improvement in higher even dimensions $d\ge 6.$
Our first main result is as follows.

\begin{theorem}\label{main1} Let $P\subset \mathbb F_q^d$ be the paraboloid defined as in \eqref{defP}.
If the dimension $d\ge 6$ is even, then for each $\varepsilon >0$ we have
$$ R_P^*\left( 2 \to  \frac{6d+8}{3d-2} +\varepsilon\right) \lesssim 1.$$
\end{theorem}

Notice that if $d\ge 6$, then $(6d+8)/(3d-2) <2d^2/(d^2-2d+2),$ which implies that Theorem \ref{main1} is better than the result $R_P^*(2\to 2d^2/(d^2-2d+2))\lesssim 1$ due to A. Lewko and M. Lewko.\\

The statement $(2)$ in Conjecture \ref{Conj2} has not been solved in any case.
In the case when $d=3$ and $q$ is a prime with $q\equiv 3 \,(\mbox{mod}~4)$, Mockenhaupt and Tao  \cite{MT04} deduced the following extension result: for every $\varepsilon >0$,
\begin{equation} \label{Ta3D} R_P^*\left(2 \to \frac{18}{5}+\varepsilon\right)\lesssim 1.\end{equation}
This was improved to $R_P^*(2 \to \frac{18}{5}) \lesssim 1$ by A. Lewko and M. Lewko \cite{LL10} (Bennett, Carbery, Garrigos, and Wright independently proved it in unpublished work).
Recently, Lewko \cite{LL13} discovered a nice connection between the finite field extension problem and the finite field Szemer\'{e}di-Trotter incidence problem.
Using the connection with ingenious arguments, he obtained the currently best known result on extension problems for the 3-d paraboloid. More precisely, he proved that
if the dimension $d$ is three and $-1\in \mathbb F_q$ is not a square, then there exists an $\varepsilon>0$ such that
\begin{equation}\label{Lew3DG} R_P^*\left(2\to \frac{18}{5}-\varepsilon\right)\lesssim 1.\end{equation}
Furthermore, assuming that $q$ is a prime and $-1\in \mathbb F_q$ is not a square, he gave the following explicit result for $d=3$:
\begin{equation} \label{Lew3D} R_P^*\left(2\to \frac{18}{5}-\frac{1}{1035}+\varepsilon\right) \lesssim 1 \quad \mbox{for any}\quad \varepsilon>0.\end{equation}
Although this result is still far from the conjectured result, $R^*_P(2\to 3)\lesssim 1,$  M. Lewko provided novel ideas useful in developing the finite field extension problem
and we will also adapt many of his methods to deduce our improved results.
In specific higher odd dimensions, Iosevich and Koh \cite{IK09} proved that $R_P^*(2\to \frac{2d^2}{d^2-2d+2}) \lessapprox 1$ with the assumptions of the statement $(2)$ in Conjecture \ref{Conj2}.
This result is also better than the Stein-Tomas inequality. A. Lewko and M. Lewko \cite{LL10} obtained the endpoint estimate so that the result by Iosevich and Koh was  improved to
\begin{equation} \label{LewR} R_P^*\left(2\to \frac{2d^2}{d^2-2d+2}\right) \lesssim 1.\end{equation}
As our second result, we shall improve this result in the case when $d=4\ell-1 \ge 7$ for $\ell \in \mathbb N.$
More precisely, we  have the following result.

\begin{theorem}\label{main2} Let $P \subset \mathbb F_q^d$ be the paraboloid defined as in \eqref{defP}.
If $d=4\ell+3$ for $\ell\in \mathbb N$, and $ -1\in \mathbb F_q$ is not a square number, then for every $\varepsilon >0$, we have
$$ R_P^*\left(2\to \frac{6d+10}{3d-1} +\varepsilon \right)\lesssim 1.$$
\end{theorem}
Notice that Theorem \ref{main2} is superior to the result \eqref{LewR} due to A. Lewko and M. Lewko.
If one could obtain the exponent in Theorem \ref{main2} for $d=3$, we could have $R^*_P(2\to \frac{7}{2}+\varepsilon)\lesssim 1,$ which is much better than the best known result \eqref{Lew3D} due to M. Lewko.
Unfortunately, our result does not cover the case of three dimensions and it only improves the previous known results in specific higher odd dimensions.\\

This paper will be organized as follows. In section 2, we deduce the necessary conditions for $R^*_P(p\to r)$ bound from which we make a conjecture on extension problems for paraboloids.
In section 3, we collect several lemmas which are essential in proving our main results, Theorem \ref{main1} and Theorem \ref{main2}.
In the final section, we give the complete proofs of our main theorems.
In addition, we shall provide summary of progress on the finite field extension problems for paraboloids.

\section{Conjecture on extension problems for paraboloids}\label{secII}

In \cite{MT04}, Mockenhaupt and Tao observed that if $|V|\sim q^{d-1},$ then the necessary conditions for $R_V^*(p\to r)\lesssim 1$ are given by
\begin{equation} \label{Necessary1}
r\geq \frac{2d}{d-1} \quad \mbox{and} \quad r\geq \frac{pd}{ (p-1)(d-1)}. \end{equation}
In particular, when the variety $V$ contains an affine subspace $\Omega$ with $|\Omega|=q^k$ for $0\le k\le d-1$, the above necessary conditions can be improved to the conditions
\begin{equation}\label{Necessary2}
r\geq \frac{2d}{d-1}  \quad \mbox{and} \quad r\geq\frac{p(d-k)}{(p-1)(d-1-k)}.\end{equation}

Now, let us observe the necessary conditions for $R^*_P(p\to r)$ bound where the paraboloid $P\subset \mathbb F_q^d$ is defined as in \eqref{defP}.
To find more exact necessary conditions for $R_P^*(p\to r)\lesssim 1,$
it is essential to know the size of subspaces lying on the paraboloid $P\subset \mathbb F_q^d.$
To this end, we need the following lemma which is a direct consequence of Lemma 2.1 in \cite{Vi12}.

\begin{lemma}\label{Vi} Let $S_0=\{(x_1,\ldots, x_{d-1})\in \mathbb F_q^{d-1}: x_1^2+\cdots+x_{d-1}^2=0\}$ be a variety in $\mathbb F_q^{d-1}$ with $d\ge 2.$
Denote by $\eta$ the quadratic character of $\mathbb F_q.$
If $W$ is a subspace of maximal dimension contained in $S_0$, then we have the following facts:
\begin{enumerate}
\item If $d-1$ is odd, then $|W|=q^{\frac{d-2}{2}}$
\item If $d-1$ is even and $(\eta(-1))^{\frac{d-1}{2}}=1$, then $|W|=q^{\frac{d-1}{2}}$
\item If $d-1$ is even and $(\eta(-1))^{\frac{d-1}{2}}=-1,$ then $|W|=q^{\frac{d-3}{2}}.$
\end{enumerate}
\end{lemma}

Observe from Lemma \ref{Vi} that $\Omega:=W \times \{0\} \subset \mathbb F_q^{d-1} \times \mathbb F_q$ is a subspace contained in the paraboloid $P \subset \mathbb F_q^d.$
Since $|\Omega|=|W|$, we have the following result from Lemma \ref{Vi}.

\begin{corollary} Let $P\subset \mathbb F_q^d$ be the paraboloid. Then the following statements hold:
\begin{enumerate} \label{SubP}
\item If $d\ge 2$ is even, then the paraboloid $P$ contains a subspace $\Omega$ with $|\Omega|=q^{\frac{d-2}{2}}$
\item If $d=4\ell-1$ for $\ell\in \mathbb N$, and $ -1\in \mathbb F_q$ is not a square number, then the paraboloid $P$ contains a subspace $\Omega$ with $|\Omega|=q^{\frac{d-3}{2}}$
\item If $d=4\ell+1$ for $\ell \in \mathbb N$, then the paraboloid $P$ contains a subspace $\Omega$ with $|\Omega|= q^{\frac{d-1}{2}}$
\item If $d\ge 3$ is odd, and $-1\in \mathbb F_q$ is a square number, then the paraboloid $P$ contains a subspace $\Omega$ with $|\Omega|= q^{\frac{d-1}{2}}.$
\end{enumerate}
\end{corollary}

Applying Corollary \ref{SubP} to \eqref{Necessary2}, the necessary conditions for $R_P^*(p\to r)\lesssim 1$ are given as follows:
\begin{lemma} \label{GeN}Let $P\subset \mathbb F_q^d$ be the paraboloid defined as in \eqref{defP}.
Assume that $R_P^*(p\to r)\lesssim 1$ for $1\le p,r\le \infty.$ Then the following statements are true:
\begin{enumerate}
\item If $d\ge 2$ is even, then $(1/p, 1/r)$ must be contained in the convex hull of points
$$(1, 0), (0,0), \left(0, \frac{d-1}{2d}\right),\, \mbox{and}~~ P_1:=\left(\frac{d^2-d+2}{2d^2},~~ \frac{d-1}{2d}\right).$$
\item If $d=4\ell-1$ for $\ell\in \mathbb N$, and $ -1\in \mathbb F_q$ is not a square number, then  $(1/p, 1/r)$ lies on the convex hull of points
$$(1, 0), (0,0), \left(0, \frac{d-1}{2d}\right), \, \mbox{and}~~ P_2:=\left(\frac{d^2+3}{2d^2+2d},~~ \frac{d-1}{2d}\right).$$

\item If  $d=4\ell+1$ for $\ell \in \mathbb N$, then  $(1/p, 1/r)$ must be contained in the convex hull of points $(1, 0), (0,0), \left(0, \frac{d-1}{2d}\right),$ and
$P_3:=\left(\frac{d-1}{2d},~~ \frac{d-1}{2d}\right).$
\item If $d\ge 3$ is odd, and $-1\in \mathbb F_q$ is a square number, then  $(1/p, 1/r)$ must be contained in the convex hull of points $(1, 0), (0,0), \left(0, \frac{d-1}{2d}\right),$ and
$\left(\frac{d-1}{2d},~~ \frac{d-1}{2d}\right).$
\end{enumerate}
\end{lemma}

We may conjecture that the necessary conditions for $R_P^*(p\to r)\lesssim 1$ in Lemma \ref{GeN} are in fact  sufficient.
For this reason, we could settle the extension problem for paraboloids if we could obtain the critical endpoints $P_1, P_2, P_3$ in the statement of Lemma \ref{GeN}.
In conclusion,  to solve the extension problem for paraboloids, it suffices to establish the following conjecture on critical endpoints.
\begin{conjecture}\label{simpleconj} The following statements hold:
\begin{enumerate}
\item If $d\ge 2$ is even, then  $ R_P^*\left(\frac{2d^2}{d^2-d+2},~~ \frac{2d}{d-1}\right)\lesssim 1$
\item If $d=4\ell-1$ for $\ell\in \mathbb N$, and $ -1\in \mathbb F_q$ is not a square number, then
$ R_P^*\left(\frac{2d^2+2d}{d^2+3},~~ \frac{2d}{d-1}\right)\lesssim 1$
\item  If $d=4\ell+1$ for $\ell \in \mathbb N$, then $ R_P^*\left(\frac{2d}{d-1},~~ \frac{2d}{d-1}\right) \lesssim 1$
\item  If $d\ge 3$ is odd, and $-1\in \mathbb F_q$ is a square number, then
$R_P^*\left(\frac{2d}{d-1},~~ \frac{2d}{d-1}\right)\lesssim 1.$
\end{enumerate}
\end{conjecture}

\section{Preliminary lemmas}

 In this section, we collect several lemmas which shall be used to prove our main results.
As we shall see, both Theorem \ref{main1} and Theorem \ref{main2} will be proved in terms of  the restriction estimates (dual extension estimate).
Thus, we start with lemmas about the restriction operators associated with paraboloids.
We shall write $R_P(p\to r)$ for $R^*_P(r'\to p')$ for $1\le p,r \le \infty.$ Namely, $R_P(p\to r)$ is the smallest positive real number such that
the following restriction estimate holds:
$$\|\widehat{g}\|_{L^{r}(P, d\sigma)} \le R_P(p\to r) \,\|g\|_{L^{p}(\mathbb F_q^d, dm)} \quad\mbox{for all functions}~~g:(\mathbb F_q^d, dm) \to \mathbb C.$$

The following definition was given in \cite{LL13}.
\begin{definition}\label{defregular} Let $G\subset \mathbb F_q^d.$ For each $a\in \mathbb F_q$,  define a level set
$$ G_a=\{(m_1,\ldots, m_{d-1}, m_d) \in G: m_d=a\}.$$
In addition, define $$ L_G=\{a\in \mathbb F_q: |G_a| \ge 1 \}.$$
We say that the set $G$ is a regular set if
$$ \frac{|G_a|}{2}\le |G_{a'}| \le 2\,|G_a| 
\quad \mbox{for}~~ a, a'\in L_G.$$
Finally, the function $g:\mathbb F_q^d \to \mathbb C$ is called  a regular function if
the function $g$ is supported on a regular set $G$ and $\frac{1}{2}\le |g(m)|\le 1$ for $m\in G.$
\end{definition}
Notice  that if $G$ is a regular set, then $|G|\sim |G_a||L_G|$ for all $a\in L_G.$
By the the dyadic pigeonhole principle, the following lemma was given by M. Lewko (see Lemma 14 in \cite{LL13}).
\begin{lemma} \label{lem3.2} If the restriction estimate
$$\|\widehat{g}\|_{L^{r}(P, d\sigma)} \le R_P(p\to r) \,\|g\|_{L^{p}(\mathbb F_q^d, dm)}$$
holds for all regular functions $g:(\mathbb F_q^d, dm)\to \mathbb C,$ then for each $\varepsilon >0$,
$$R_P\left(p-\varepsilon\, \to r\right) \lesssim 1.$$
\end{lemma}

Working on regular test functions, we lose the endpoint result but our analysis becomes extremely simplified.
When the size of the support $G$ of a regular function $g$ is somewhat big,  we shall invoke the following restriction estimate.
\begin{lemma} \label{lem3.3} Let $g$ is a regular function on $(\mathbb F_q^d, dm)$ with $\mbox{supp}(g)=G.$
Then we have
$$\|\widehat{g}\|_{L^2(P,d\sigma)} \le q^{\frac{1}{2}} |G|^{\frac{1}{2}}.$$
\end{lemma}
\begin{proof} By the Plancherel theorem,  we see that
$$ \|{(fd\sigma)}^\vee\|_{L^2(\mathbb F_q^d, dm)} = q^{\frac{1}{2}} \|f\|_{L^2(P, d\sigma)} \quad \mbox{for all functions}~~f: P\to \mathbb C.$$
By duality, it is clear that
$$ \|\widehat{g}\|_{L^2(P,d\sigma)} \le q^{\frac{1}{2}}\|g\|_{L^2( \mathbb F_q^d, dm)} \le q^{\frac{1}{2}}  \|G\|_{L^2( \mathbb F_q^d, dm)} = q^{\frac{1}{2}} |G|^{\frac{1}{2}},$$
where the last inequality follows from the property of the regular function $g$ (namely, $\frac{1}{2}\le |g|\le 1$ on its support $G$.)
\end{proof}

The following result is well known in \cite{MT04} (see also \cite{IK09}).

\begin{lemma}\label{explicit}
Let $d\sigma$ be the normalized surface measure on the paraboloid $P \subset (\mathbb F_q^d, d\xi).$
For each $m=(\underline{m}, m_d) \in {\mathbb F}_q^{d-1}\times {\mathbb F}_q$ , we have
$$ (d\sigma)^{\vee}(m)= \left\{\begin{array}{ll} q^{-(d-1)} \chi \left( \frac{\|\underline{m}\| }{-4m_d}\right)
\eta^{d-1}(m_d)\, G_1^{d-1}
 \quad &\mbox{if} \quad m_d \ne 0\\
0 \quad &\mbox{if} \quad m_d =0,\, m \ne (0,\dots,0)\\
1 \quad &\mbox{if} \quad m=(0,\ldots,0).\end{array}\right.,$$
where $\|\underline{m}\|:=m_1^2+\cdots+ m_{d-1}^2$, $\eta$ denotes the quadratic character of $\mathbb F_q^*$, and $ G_1$ denotes the standard Gauss sum with $|G_1|=|\sum\limits_{s\ne 0} \eta(s) \chi(s)|=q^{\frac{1}{2}}.$
\end{lemma}

When  a regular function $g$ is supported on a small set $G$, the following result will be useful to deduce a good $L^2$ restriction estimate.
\begin{lemma}\label{lem3.5} If $g$ is a regular function on $(\mathbb F_q^d, dm)$ with $\mbox{supp}(g)=G,$ then we have
$$\|\widehat{g}\|_{L^2(P, d\sigma)} \lesssim |G|^{\frac{1}{2}} + q^{\frac{-d+1}{4}} |G|.$$
\end{lemma}

\begin{proof} It follows that
\begin{align*}
\|\widehat{g}\|^2_{L^2(P, d\sigma)} &=\frac{1}{|P|} \sum_{\xi \in P} |\widehat{g}(\xi)|^2=\frac{1}{q^{d-1}} \sum_{\xi\in P} \sum_{m, m'\in G} \chi(\xi\cdot(m-m')) g(m) \overline{g(m')}\\
                                  &=q \sum_{m,m'\in G} {P}^\vee(m-m')g(m) \overline{g(m')} \le q \sum_{m, m'\in G} |{P}^\vee(m-m')|\\
                                  &=q \sum_{m\in G} |{P}^\vee(0,\ldots,0)| + q \sum_{m,m'\in G: m\ne m'} |{P}^\vee(m-m')|= \mbox{I} + \mbox{II}.
\end{align*}
Since ${P}^\vee(0,\ldots,0) =\frac{|P|}{q^d}=\frac{1}{q},$ we see that $\mbox{I}=|G|.$ To estimate $\mbox{II}$, we observe from Lemma \ref{explicit} that  if $w\ne (0,\ldots,0),$
$$ |{P}^\vee(w)| = \left|\frac{1}{q} \,(d\sigma)^\vee(w)\right|\le q^{\frac{-d-1}{2}}.$$
Then it is clear that $\mbox{II}\le q^{\frac{-d+1}{2}} |G|^2.$  Putting all estimates together, we obtain the lemma.
\end{proof}

The improved $L^p\to L^2$ restriction estimates for paraboloids have been obtained by extending the idea of Carbery \cite{Ca92} to the finite field setting.
For instance, Mockenhaupt and Tao \cite{MT04} observed that the restriction operator acting on a single vertical slice of g, say $g_a$ for $a\in \mathbb F_q,$
is closely related to the extension operator applied to a function $h$ on $P$, which can be identified with the slice function $g_a.$
In fact, they found the connection between the $L^p\to L^2$ restriction estimate and the $L^p\to L^4$ extension estimate  obtained from the additive energy estimation. Recall that the additive energy $\Lambda(E)$ for $E\subset P$ is given by
\begin{equation}\label{additive} \Lambda(E):= \sum_{x,y,z, w\in E: x+y=z+w} 1.\end{equation}
As a consequence, they obtained the extension result \eqref{Ta3D} for the 3-D paraboloid.
Working with the restriction operator applied to regular test functions,   M. Lewko \cite{LL13} was able to achieve the further improved extension results  for the 3-$D$ paraboloid (see \eqref{Lew3DG} and \eqref{Lew3D}).
He also employed the relation between the $L^p\to L^2$ restriction estimate and the $L^p\to L^4$ extension result for the 3-D paraboloid.
In this paper, we develop his work to higher dimensional cases.
To estimate $\|\widehat{g}\|_{L^2(P, d\sigma)}$,  we will invoke not only $L^p\to L^4$ extension results but also $L^2\to L^r$ extension results for paraboloids in higher dimensions.
The following lemma can be obtained by a modification of the Mockenhaupt and Tao Machinery which explains the relation between the $L^p\to L^2$ restriction estimate and the $L^p\to L^4$ extension result for paraboloids.

\begin{lemma}\label{key1} Let $P\subset \mathbb F_q^d$ be the paraboloid. Then the following statements hold:
\begin{enumerate}
\item Let $g$ be a regular function with the support $G\subset (\mathbb F_q^d, dm).$ For each $a\in L_G,$ let $h_a$ be a function on the paraboloid $P\subset (\mathbb F_q^d, d\xi)$ such that $\frac{1}{2} \le |h_a(\xi)|\le 1$ on
$\mbox{supp}(h_a)$ and $|\mbox{supp}(h_a)|= |G_a|.$ In addition, assume that there exists  a positive number $U(|E|)$ depending on the size of a set $E\subset P$ such that $|E|\sim |\mbox{supp}(h_a)|$ for all $a\in L_G$ and
\begin{equation}\label{assumption1} \max_{a\in L_G} \|(h_a d\sigma)^\vee\|_{L^4(F_q^d, dm)} \lesssim U(|E|).\end{equation} Then we have
$$ \|\widehat{g}\|_{L^2(P, d\sigma)} \lesssim |G|^{\frac{1}{2}} + |G|^{\frac{3}{8}}\, |L_G|^{\frac{1}{2}} q^{\frac{d-1}{4}} (U(|E|)^{\frac{1}{2}}.$$
\item  If $d\ge 4$ is even, or if $d=4\ell+3$ for $\ell\in \mathbb N$ and $ -1\in \mathbb F_q$ is not a square number, then
$$  \|\widehat{g}\|_{L^2(P, d\sigma)} \lesssim |G|^{\frac{d^2+d-1}{2d^2}}|L_G|^{\frac{1}{4}}$$
for all regular functions $g$ on $(\mathbb F_q^d, dm)$ with $\mbox{supp}(g)=G.$
\end{enumerate}
\end{lemma}
\begin{proof} By duality, it follows that
$$  \|\widehat{g}\|^2_{L^2(P, d\sigma)} = <g,\, (\widehat{g} d\sigma)^\vee> =<g,\, g\ast (d\sigma)^\vee>.$$
Using the Bochner-Riesz kernel $K$ which is defined by
$K(m)= (d\sigma)^\vee(m) -\delta_0(m)$ for $m\in (\mathbb F_q^d, dm),$ where $\delta_0(m)=1$ if $m=(0, \ldots,0)$ and $0$ otherwise,
we can write from H\"{o}lder's inequality that for $1\le r\le \infty,$
\begin{align} \label{L2ofg} \|\widehat{g}\|^2_{L^2(P, d\sigma)} &= <g, \, g\ast \delta_0> + <g,\,g\ast K> \\
\nonumber &\le\|g\|^2_{L^2(\mathbb F_q^d, dm)} + \|g\|_{L^{r'}(\mathbb F_q^d, dm)} \,\| g\ast K\|_{L^r(\mathbb F_q^d, dm)}\\
\nonumber &\le  |G| + |G|^{\frac{1}{r'}} \, \| g\ast K\|_{L^r(\mathbb F_q^d, dm)},\end{align}
where the last inequality follows from the property of a regular function $g$ with $\frac{1}{2}\le g\le 1$ on its support $G.$
To estimate $\| g\ast K\|_{L^r(\mathbb F_q^d, dm)},$ define $g_a$ for $a\in L_G$ as the restriction of $g$ to the hyperplane $\{m=(m_1, \ldots, m_d)\in \mathbb F_q^d: m_d=a\}.$
Notice that $\mbox{supp}(g_a)=G_a$ for $a\in L_G.$ It follows that
\begin{equation} \label{mLG}\| g\ast K\|_{L^r(\mathbb F_q^d, dm)} \le \sum_{a\in L_G} \|g_a\ast K\|_{L^r(\mathbb F_q^d, dm)}.\end{equation}
By the definition of $K$ and Lemma \ref{explicit}, we see that for each $a\in L_G,$
\begin{align*} \|g_a\ast K\|_{L^r(\mathbb F_q^d, dm)} &= \left(\sum_{m\in \mathbb F_q^d} \left| \sum_{n\in \mathbb F_q^d} g_a(n) K(m-n)\right|^r\right)^{\frac{1}{r}}\\
 &= q^{\frac{-d+1}{2}} \left( \sum_{\underline{m} \in \mathbb F_q^{d-1}} \sum_{m_d\ne a} \left| \sum_{\underline{n}\in \mathbb F_q^{d-1}} g(\underline{n}, a)
  \,\chi\left(\frac{\|\underline{m}-\underline{n}\|}{-4(m_d-a)}\right)\right|^r\right)^{\frac{1}{r}},
  \end{align*}
where we define $\|\underline{m}-\underline{n}\|=(\underline{m}-\underline{n})\cdot (\underline{m}-\underline{n}).$ 
After changing variables by letting $s=-m_d+a,$ we use the change of variables one more by putting
$ t=\frac{1}{4s}$ and $\underline{u}=\frac{-\underline{m}}{2s}.$ Then it follows that
\begin{align*}\|g_a\ast K\|_{L^r(\mathbb F_q^d, dm)}&=q^{\frac{-d+1}{2}}\left( \sum_{\underline{u} \in \mathbb F_q^{d-1}} \sum_{t\ne 0} \left|\chi\left(\frac{\underline{u}\cdot \underline{u} }{4t} \right) \sum_{\underline{n}\in \mathbb F_q^{d-1}} g(\underline{n}, a)
 \,\chi\left((\underline{u}\cdot \underline{n})+t \,\underline{n}\cdot \underline{n}) \right)\right|^r \right)^{\frac{1}{r}}\\
&=q^{\frac{-d+1}{2}}\left( \sum_{\underline{u} \in \mathbb F_q^{d-1}} \sum_{t\ne 0} \left| \sum_{\underline{n}\in \mathbb F_q^{d-1}} g(\underline{n}, a)
 \,\chi\left((\underline{u}, t)\cdot (\underline{n}, \,\underline{n}\cdot \underline{n}) \right)\right|^r \right)^{\frac{1}{r}}.\end{align*}
Now, for each $a\in L_G,$ define $h_a$ as a function on the paraboloid $P$ given by
\begin{equation}\label{relation} h_a(\underline{n}, \,\underline{n}\cdot \underline{n}) = g_a(n) =g(\underline{n}, a) \quad \mbox{for}~~ n=(\underline{n}, n_d) \in \mathbb F_q^{d-1} \times \mathbb F_q.\end{equation}
Then we see that for each $a\in L_G,$
$$\|g_a\ast K\|_{L^r(\mathbb F_q^d, dm)} \le q^{\frac{d-1}{2}} \| (h_a d\sigma)^\vee\|_{L^r(\mathbb F_q^d, dm)}.$$
Hence, combining this with \eqref{mLG}, the inequality \eqref{L2ofg} implies that
\begin{equation}\label{L2formula}
\|\widehat{g}\|_{L^2(P, d\sigma)} \lesssim |G|^{\frac{1}{2}} + |G|^{\frac{1}{2r'}} q^{\frac{d-1}{4}} \left(\sum_{a\in L_G} \|(h_ad\sigma)^\vee \|_{L^r(\mathbb F_q^d, dm)}\right)^{\frac{1}{2}}.
\end{equation}
\subsection{Proof of the statement (1) in Lemma \ref{key1}}
Since $g$ is a regular function supported on the regular set $G,$ it is clear from the definition of $h_a$ that  $\frac{1}{2}\le |h_a(\xi)|\le 1$
on $\mbox{supp}(h_a)$ and  $|\mbox{supp}(h_a)|=|\mbox{supp}(g_a)|=|G_a|$ for $a\in L_G.$ Thus,  using the assumption \eqref{assumption1} with $r=4$, the inequality \eqref{L2formula} gives the desirable conclusion.
\subsection{Proof of the statement (2) in Lemma \ref{key1}}
We shall appeal the following $L^2\to L^r$ extension result obtained by A. Lewko and M. Lewko (see Theorem 2 in \cite{LL10}).
\begin{lemma}\label{LLL} Let $P$ be the paraboloid in $(\mathbb F_q^d, d\xi).$ If $d\ge 4$ is even, or if $d=4\ell+3$ for $\ell\in \mathbb N$ and $ -1\in \mathbb F_q$ is not a square number, then we have
$$ R^*_P\left(2\to \frac{2d^2}{d^2-2d+2}\right) \lesssim 1.$$
\end{lemma}
Applying this lemma to the inequality \eqref{L2formula} with $r=\frac{2d^2}{d^2-2d+2},$ it follows
$$ \|\widehat{g}\|_{L^2(P, d\sigma)} \lesssim |G|^{\frac{1}{2}} + |G|^{\frac{d^2+2d-2}{4d^2}} q^{\frac{d-1}{4}} \left(\sum_{a\in L_G} \|h_a \|_{L^2(P, d\sigma)}\right)^{\frac{1}{2}}.$$
By the Cauchy-Schwarz inequality and the definition of $h_a$ given in \eqref{relation}, we conclude that
\begin{align*}
 \|\widehat{g}\|_{L^2(P, d\sigma)} &\lesssim |G|^{\frac{1}{2}} + |G|^{\frac{d^2+2d-2}{4d^2}} q^{\frac{d-1}{4}} |L_G|^{\frac{1}{4}}  \left( \sum_{a\in L_G}  \|h_a \|^2_{L^2(P, d\sigma)}\right)^{\frac{1}{4}}\\
  &=|G|^{\frac{1}{2}} + |G|^{\frac{d^2+2d-2}{4d^2}} q^{\frac{d-1}{4}} |L_G|^{\frac{1}{4}}  \left( \sum_{a\in L_G} \frac{1}{q^{d-1}} \sum_{n\in P} |h_a(n)|^2 \right)^{\frac{1}{4}}\\
  &= |G|^{\frac{1}{2}} + |G|^{\frac{d^2+2d-2}{4d^2}} |L_G|^{\frac{1}{4}} \left( \sum_{a\in L_G}\sum_{n\in \mathbb F_q^d} |g_a(n)|^2 \right)^{\frac{1}{4}}\\
  &= |G|^{\frac{1}{2}} + |G|^{\frac{d^2+2d-2}{4d^2}} |L_G|^{\frac{1}{4}} \left(\sum_{n\in \mathbb F_q^d} |g(n)|^2 \right)^{\frac{1}{4}}\\
  &\le |G|^{\frac{1}{2}} + |G|^{\frac{d^2+2d-2}{4d^2}} |L_G|^{\frac{1}{4}} |G|^{\frac{1}{4}} \lesssim |G|^{\frac{d^2+d-1}{2d^2}}|L_G|^{\frac{1}{4}},
\end{align*}
where the last line follows because $ \frac{1}{2} \le |g(n)| \le 1$ on its support $G.$
\end{proof}

\section{Proof of main theorems}
First, let us see basic ideas to deduce our main results.
We want to improve  Lemma \ref{LLL} which is the previously best known result on extension problems for paraboloids in higher dimensions. By duality, Lemma \ref{LLL} implies the following restriction estimate:
\begin{equation}\label{oldresult} \|\widehat{g}\|_{L^2(P, d\sigma)} \lesssim
\|g\|_{L^{\frac{2d^2}{d^2+2d-2}}(\mathbb F_q^d, dm)}.\end{equation}
Now let us only consider the regular function $g$ on its support $G.$
Since $\|g\|_{L^p(\mathbb F_q^d, dm)} \sim |G|^{\frac{1}{p}}$,
when $|G|$ is much bigger than $q^{\frac{d^2}{2d-2}}$,  Lemma \ref{lem3.3} already gives us a better result than \eqref{oldresult}.  On the other hand, 
when $|G|$ is very small, Lemma \ref{lem3.5} yields very strong results.
Therefore, our main task is to obtain much better estimate than \eqref{oldresult} for every set $G$ with $q^{\frac{d^2}{2d-2}-\delta} \le |G| \le q^{\frac{d^2}{2d-2}+\varepsilon}$ for some $\delta,\,\varepsilon>0.$
This will be successfully done by applying  Lemma \ref{key1}. 
In practice, we need to find a $U(|E|)$ in the conclusion of the first part of Lemma \ref{key1}. To do this,  we shall invoke the following additive energy estimates due to Iosevich and Koh (see Lemma 7, Lemma 8, and Remark 4 in \cite{IK09}).
\begin{lemma}\label{key} Let $P$ be the paraboloid in $({\mathbb F}_q^d , d\xi).$
Then the following statements hold:
\begin{enumerate}
\item If the dimension $d\ge 4$ is even and $E\subset P$, then we have
$$\Lambda(E) \lesssim \min \{ |E|^3,~~ q^{-1}|E|^3+q^{\frac{d-2}{4}}|E|^{\frac{5}{2}} + q^{\frac{d-2}{2}}|E|^2 \}$$

\item If $d=4\ell+3$ for $\ell\in \mathbb N$, and $ -1\in \mathbb F_q$ is not a square number, then we have
$$\Lambda_4(E) \lesssim \min \{ |E|^3,~~ q^{-1}|E|^3+q^{\frac{d-3}{4}}|E|^{\frac{5}{2}} + q^{\frac{d-2}{2}}|E|^2 \},$$
where $\Lambda(E)$ denotes the additive energy defined as in \eqref{additive}.
\end{enumerate}
\end{lemma}

As we shall see, we only need the upper bound of $\Lambda(E)$ for a restricted range of $E \subset P.$ Considering the dominating value in terms of $|E|$, the following result is a simple corollary of the lemma above.

\begin{corollary}\label{cor1}
For the paraboloid $P \subset (\mathbb F_q^d, d\xi),$  we have the following facts:
\begin{enumerate}
\item If the dimension $d\ge 4$ is even and $E$ is any subset of $P$ with $q^{\frac{d-2}{2}} \le |E|\le q^{\frac{d+2}{2}},$ then
$$\Lambda(E) \lesssim q^{\frac{d-2}{4}}|E|^{\frac{5}{2}}$$

\item Suppose that $d=4\ell+3$ for $\ell\in \mathbb N$, and $ -1\in \mathbb F_q$ is not a square number.
Then, for any subset $E$ of $P$ with $q^{\frac{d-2}{2}} \le |E|\le q^{\frac{d+1}{2}},$ we have
$$\Lambda(E) \lesssim q^{\frac{d-3}{4}}|E|^{\frac{5}{2}} + q^{\frac{d-2}{2}}|E|^2.$$
\end{enumerate}
\end{corollary}

We can deduce the following result by applying Corollary \ref{cor1} to the first part of Lemma \ref{key1}.
\begin{lemma} \label{lem4.3}
Let $g$ be a regular function with its support $G\subset (\mathbb F_q^d, dm).$ Then the following statements are valid:
\begin{enumerate}
\item If the dimension $d\ge 4$ is even and $q^{\frac{d-2}{2}} \lesssim |G_a| \lesssim q^{\frac{d+2}{2}}$ for $a\in L_G,$ then we have
 $$ \|\widehat{g}\|_{L^2(P, d\sigma)} \lesssim |G|^{\frac{1}{2}} + |G|^{\frac{11}{16}}\, |L_G|^{\frac{3}{16}} q^{\frac{-3d+6}{32}}$$
 \item
 Assume that $d=4\ell+3$ for $\ell\in \mathbb N$, and $ -1\in \mathbb F_q$ is not a square number. Then if
 $q^{\frac{d-2}{2}} \lesssim |G_a| \lesssim q^{\frac{d+1}{2}}$ for $a\in L_G$, we have
 $$ \|\widehat{g}\|_{L^2(P, d\sigma)} \lesssim |G|^{\frac{1}{2}} +
 |G|^{\frac{11}{16}} |L_G|^{\frac{3}{16}} q^{\frac{-3d+5}{32}}+|G|^{\frac{5}{8}} |L_G|^{\frac{1}{4}} q^{\frac{-d+2}{16}}.$$
\end{enumerate}
\end{lemma}

\begin{proof} For each $a\in L_G$, let $h_a$ be the function on $P$ given in the statement $(1)$ of Lemma \ref{key1}.
For each $a\in L_G,$ let $H_a=\mbox{supp}(h_a).$
Since $\frac{1}{2} \le |h_a|\le 1$ on its support $H_a,$  expanding $L^4$ norm of $(h_ad\sigma)^\vee$ gives
$$ \|(h_a d\sigma)^\vee\|_{L^4(F_q^d, dm)} \le \| (H_a d\sigma)^\vee\|_{L^4(F_q^d, dm)} = q^{\frac{-3d+4}{4}} (\Lambda(H_a))^{\frac{1}{4}}. $$
First, let us prove the first part of Lemma \ref{lem4.3}. Since $|G_a|=|H_a|$ for $a\in L_G,$  the first part of Corollary \ref{cor1} and the above inequality  yield
$$ \|(h_a d\sigma)^\vee\|_{L^4(F_q^d, dm)} \lesssim q^{\frac{-3d+4}{4}} \left(q^{\frac{d-2}{4}}|H_a|^{\frac{5}{2}}\right)^{\frac{1}{4}}=q^{\frac{-11d+14}{16}} |H_a|^{\frac{5}{8}}.$$
By the definition of a regular set $G,$ it is obvious that $|G_a|\sim |G_{a'}|$ for $a, a' \in L_G.$ Hence, $|H_a|\sim |H_{a'}|$ for $a, a' \in L_G.$ Thus, we can choose $E\subset P$ such that
$|E|\sim |H_a|$ for all $a\in L_G.$ It follows that
$$ \max_{a\in L_G} \|(h_a d\sigma)^\vee\|_{L^4(F_q^d, dm)} \lesssim q^{\frac{-11d+14}{16}} |E|^{\frac{5}{8}}:=U(|E|).$$
By applying the first part of Lemma \ref{key1} and observing that $|G|\sim |G_a||L_G|\sim |E||L_G|$ for all $a\in L_G,$ we conclude that
\begin{align*}\|\widehat{g}\|_{L^2(P, d\sigma)} &\lesssim |G|^{\frac{1}{2}} + |G|^{\frac{3}{8}}\, |L_G|^{\frac{1}{2}} q^{\frac{d-1}{4}} \left(q^{\frac{-11d+14}{16}} |E|^{\frac{5}{8}}\right)^{\frac{1}{2}}\\
&\sim |G|^{\frac{1}{2}} + |G|^{\frac{11}{16}}\, |L_G|^{\frac{3}{16}} q^{\frac{-3d+6}{32}},\end{align*}
which proves the first part of Lemma \ref{lem4.3}. \\

To prove the second part of Lemma \ref{lem4.3},  we use the same arguments as in the proof of the first part of Lemma \ref{lem4.3}. In this case, we just utilize the second part of Corollary \ref{cor1} to see that

\begin{align*} \max_{a\in L_G} \|(h_a d\sigma)^\vee\|_{L^4(F_q^d, dm)} &\lesssim q^{\frac{-3d+4}{4}}\left( q^{\frac{d-3}{4}}|E|^{\frac{5}{2}} + q^{\frac{d-2}{2}}|E|^2\right)^{\frac{1}{4}}\\
&\sim q^{\frac{-3d+4}{4}}\left( q^{\frac{d-3}{16}}|E|^{\frac{5}{8}} + q^{\frac{d-2}{8}}|E|^{\frac{1}{2}}\right)\\
&=q^{\frac{-11d+13}{16}}|E|^{\frac{5}{8}} + q^{\frac{-5d+6}{8}}|E|^{\frac{1}{2}}
:=U(|E|).\end{align*}

As before, we appeal the first part of Lemma \ref{key1} and use  that $|G|\sim |G_a||L_G|\sim |E||L_G|$ for all $a\in L_G.$ Then the proof of the second part of Lemma \ref{lem4.3} is complete as follows:
\begin{align*}
\|\widehat{g}\|_{L^2(P, d\sigma)} &\lesssim |G|^{\frac{1}{2}} + |G|^{\frac{3}{8}}\, |L_G|^{\frac{1}{2}} q^{\frac{d-1}{4}} \left(q^{\frac{-11d+13}{16}}|E|^{\frac{5}{8}} + q^{\frac{-5d+6}{8}}|E|^{\frac{1}{2}}\right)^{\frac{1}{2}}\\
                                  &\sim |G|^{\frac{1}{2}} + |G|^{\frac{3}{8}}\, |L_G|^{\frac{1}{2}} q^{\frac{d-1}{4}} \left(q^{\frac{-11d+13}{32}}|E|^{\frac{5}{16}} + q^{\frac{-5d+6}{16}}|E|^{\frac{1}{4}}   \right)\\
                                  &=|G|^{\frac{1}{2}} + |G|^{\frac{3}{8}}\, |L_G|^{\frac{1}{2}} q^{\frac{d-1}{4}} q^{\frac{-11d+13}{32}}|E|^{\frac{5}{16}} + |G|^{\frac{3}{8}}\, |L_G|^{\frac{1}{2}} q^{\frac{d-1}{4}}q^{\frac{-5d+6}
                                  {16}}|E|^{\frac{1}{4}}\\
                                  &=|G|^{\frac{1}{2}} + |G|^{\frac{3}{8}}\, |L_G|^{\frac{1}{2}}|E|^{\frac{5}{16}} q^{\frac{-3d+5}{32}} +   |G|^{\frac{3}{8}}\, |L_G|^{\frac{1}{2}}|E|^{\frac{1}{4}}q^{\frac{-d+2}{16}}\\
                                  &\sim |G|^{\frac{1}{2}} +|G|^{\frac{11}{16}}\, |L_G|^{\frac{3}{16}} q^{\frac{-3d+5}{32}} + |G|^{\frac{5}{8}}\, |L_G|^{\frac{1}{4}}q^{\frac{-d+2}{16}}.
\end{align*}
\end{proof}
We are ready to complete the proof of our main theorems, Theorem \ref{main1} and Theorem \ref{main2}, which will be proved in the following subsections.
\subsection{Proof of Theorem \ref{main1}}
By duality and Lemma \ref{lem3.2}, it is enough to prove the following statement:
\begin{theorem}\label{main1-1}If the dimension $d\ge 6$ is even, then we have
$$ \|\widehat{g}\|_{L^2(P, d\sigma)} \lesssim \|g\|_{L^{\frac{6d+8}{3d+10}}(\mathbb F_q^d, dm)}$$
for every regular function $g$ supported on $G \subset (\mathbb F_q^d, dm).$
\end{theorem}
\begin{proof} As mentioned in the beginning of this section, it is helpful to work on  three kinds of  regular functions $g$ classified according to the following size of $G=\mbox{supp}(g):$ for some $\varepsilon, \delta>0,$
$$ (1)~~  1\le |G| \le q^{\frac{d^2}{2d-2}-\delta} \quad  (2)~~ q^{\frac{d^2}{2d-2}-\delta} \le |G| \le q^{\frac{d^2}{2d-2}+\varepsilon} \quad (3)~~ q^{\frac{d^2}{2d-2}+\varepsilon}\le |G| \le q^d.$$
Notice that Lemma \ref{lem3.2} yields much strong restriction inequality whenever $|G|$ becomes lager. Thus, Lemma \ref{lem3.2} is useful for the case (3). Also observe that Lemma \ref{lem3.5} gives the better restriction inequality for  smaller size of $G$ and so it is helpful for the case (1). Thus, choosing big $\varepsilon$ and $\delta$ will yield  good results for both the case (1) and the case (3). However, whenever $\varepsilon$ and $\delta$ become larger,  the restriction estimate will be worse for the case (2).  Hence, to deduce desirable results for all cases, our main task is to select optimal values of $\varepsilon$ and $\delta.$
Now, let us see how to find the optimal $\varepsilon$ and $\delta.$ Let $\varepsilon, \delta>0$ which will be chosen later. Let $g$ be a regular function with its support $G$ such that
\begin{equation} \label{size1} q^{\frac{d^2}{2d-2}-\delta} \le |G| \le  q^{\frac{d^2}{2d-2}+\varepsilon}.\end{equation}
Let $|L_G|=q^\alpha$ for $0\le \alpha \le 1.$ Since $|G|\sim |G_a||L_G|=|G_a| q^\alpha$ for $a\in L_G,$  it must follow that for every $a\in L_G,$
$$ q^{\frac{d^2}{2d-2}-\delta-\alpha} \lesssim |G_a| \lesssim  q^{\frac{d^2}{2d-2}+\varepsilon-\alpha}.$$
In order to use the first part of Lemma \ref{lem4.3}, we need to choose $\varepsilon, \delta>0$ such that
$$q^{\frac{d-2}{2}}\le  q^{\frac{d^2}{2d-2}-\delta-\alpha} \lesssim |G_a| \lesssim  q^{\frac{d^2}{2d-2}+\varepsilon-\alpha} \le q^{\frac{d+2}{2}}.$$
Thus, if we select $\varepsilon, \delta>0$ satisfying that
\begin{equation}\label{conditione-d} \delta+\alpha \le \frac{3d-2}{2d-2} \quad \mbox{and}\quad  \varepsilon-\alpha\le \frac{d-2}{2d-2},\end{equation}
then the first part of Lemma \ref{lem4.3} yields
\begin{equation}\label{middle1} \|\widehat{g}\|_{L^2(P, d\sigma)} \lesssim |G|^{\frac{1}{2}} + |G|^{\frac{11}{16}}\,  q^{\frac{-3d+12}{32}} \quad \mbox{for}~~  q^{\frac{d^2}{2d-2}-\delta} \le |G| \le  q^{\frac{d^2}{2d-2}+\varepsilon},\end{equation}
where we use the fact that $|L_G|\le q.$ Notice that this inequality gives  worse restriction results whenever $|G|$ becomes lager. Thus, comparing this inequality with Lemma \ref{lem3.3} which gives better restriction inequality for big size of $G,$ it is desirable to choose a possibly large $\varepsilon>0$ such that
$$ |G|^{\frac{1}{2}} + |G|^{\frac{11}{16}}\,  q^{\frac{-3d+12}{32}}  \lesssim |G|^{\frac{1}{2}} q^{\frac{1}{2}}\, \left(\mbox{namely,}~ |G|\lesssim q^{\frac{3d+4}{6}}\right) \quad \mbox{and}\quad   |G| \le  q^{\frac{d^2}{2d-2}+\varepsilon}.$$
For this reason, we take $\varepsilon =\frac{d-4}{6d-6}$ which is positive for even $d\ge 6.$ Then we can take $\delta=\frac{d}{2d-2}$ so that the inequality  \eqref{conditione-d}  holds for all $0\le \alpha \le 1.$
Now we start proving Theorem \ref{main1-1}.\\

\noindent {\bf (Case I)} Assume that $q^{\frac{d}{2}} \le |G| \le q^{\frac{3d+4}{6}},$ which is the case in \eqref{size1} for $\varepsilon =\frac{d-4}{6d-6}$ and  $\delta=\frac{d}{2d-2}.$
Then, by \eqref{middle1}, we see that
$$ \|\widehat{g}\|_{L^2(P, d\sigma)} \lesssim |G|^{\frac{1}{2}} + |G|^{\frac{11}{16}}\,  q^{\frac{-3d+12}{32}} \quad \mbox{for}~~  q^{\frac{d}{2}} \le |G| \le q^{\frac{3d+4}{6}}.$$
By the direct comparison, it follows that for all $ q^{\frac{d}{2}} \le |G| \le q^{\frac{3d+4}{6}},$
$$ |G|^{\frac{1}{2}} + |G|^{\frac{11}{16}}\,  q^{\frac{-3d+12}{32}} \lesssim |G|^{\frac{3d+10}{6d+8}}=\|G\|_{L^{\frac{6d+8}{3d+10}}(\mathbb F_q^d, dm)}  \sim \|g\|_{L^{\frac{6d+8}{3d+10}}(\mathbb F_q^d, dm)}.$$
Thus, the statement of Theorem \ref{main1-1} is valid for all regular functions $g$ on $(\mathbb F_q^d, dm)$ such that $q^{\frac{d}{2}} \le |\mbox{supp}(g)|=|G| \le q^{\frac{3d+4}{6}}.$\\

\noindent {\bf (Case II)} Assume that $1 \le |G| \le q^{\frac{d}{2}}.$
Applying Lemma \ref{lem3.5}, we obtain that
$$ \|\widehat{g}\|_{L^2(P, d\sigma)} \lesssim |G|^{\frac{1}{2}} + q^{\frac{-d+1}{4}} |G| \quad\mbox{for all} ~~1 \le |G| \le q^{\frac{d}{2}}.$$
In fact, this inequality gives much stronger restriction estimate than Theorem \ref{main1-1} for $1 \le |G| \le q^{\frac{d}{2}}.$
By the direct comparison, if $1\le |G| \le q^{\frac{d}{2}},$ then we have
$$ |G|^{\frac{1}{2}} + q^{\frac{-d+1}{4}} |G| \lesssim |G|^{\frac{d+1}{2d}} = \|G\|_{L^{\frac{2d}{d+1}}(\mathbb F_q^d, dm)} \le \|G\|_{L^{\frac{6d+8}{3d+10}}(\mathbb F_q^d, dm)}  \sim \|g\|_{L^{\frac{6d+8}{3d+10}}(\mathbb F_q^d, dm)}.$$
Hence, Theorem \ref{main1-1} is proved in this case.\\

\noindent {\bf (Case III)} Finally, assume that  $ q^{\frac{3d+4}{6}}\le |G| \le q^d.$
In this case, by Lemma \ref{lem3.3} and the direct comparison, the statement of Theorem \ref{main1-1} holds:  for all $q^{\frac{3d+4}{6}}\le |G| \le q^d,$
$$\|\widehat{g}\|_{L^2(P, d\sigma)} \lesssim |G|^{\frac{1}{2}} q^{\frac{1}{2}} \lesssim \|G\|_{L^{\frac{6d+8}{3d+10}}(\mathbb F_q^d, dm)}  \sim \|g\|_{L^{\frac{6d+8}{3d+10}}(\mathbb F_q^d, dm)}.$$
We has completed the proof.
\end{proof}

\subsection{Proof of Theorem \ref{main2}}
Theorem \ref{main2} can be proved by following the same arguments as in the proof of Theorem \ref{main1}
but we will need additional work to deal with a regular set $G$ with middle size. The second part of Lemma \ref{key1} will make a crucial role in overcoming the problem. Now we start proving Theorem \ref{main2}.
By duality and Lemma \ref{lem3.2}, it suffices to prove the following statement:
\begin{theorem}\label{main2-2}
If $d=4\ell+3$ for $\ell\in \mathbb N$, and $ -1\in \mathbb F_q$ is not a square number, then we have
$$ \|\widehat{g}\|_{L^2(P, d\sigma)} \lesssim \|g\|_{L^{\frac{6d+10}{3d+11}}(\mathbb F_q^d, dm)}$$
for every regular function $g$ supported on $G \subset (\mathbb F_q^d, dm).$
\end{theorem}
\begin{proof} As in the proof of Theorem \ref{main1-1}, let $g$ be a regular function supported on the set $G\subset (\mathbb F_q^d, dm)$ satisfying that
\begin{equation} \label{size2} q^{\frac{d^2}{2d-2}-\delta} \le |G| \le  q^{\frac{d^2}{2d-2}+\varepsilon}\end{equation}
for some $\varepsilon, \delta>0$ which shall be selected as constants.
Let $|L_G|=q^\beta$ for $0\le \beta \le 1.$ Since $|G|\sim |G_a||L_G|=|G_a| q^\beta$ for $a\in L_G,$ it follows that for every $a\in L_G,$
$$ q^{\frac{d^2}{2d-2}-\delta-\beta} \lesssim |G_a| \lesssim  q^{\frac{d^2}{2d-2}+\varepsilon-\beta}.$$
For such $\varepsilon, \delta>0$, assume that for every $a\in L_G,$
$$q^{\frac{d-2}{2}}\le  q^{\frac{d^2}{2d-2}-\delta-\beta} \lesssim |G_a| \lesssim  q^{\frac{d^2}{2d-2}+\varepsilon-\beta} \le q^{\frac{d+1}{2}}.$$
Namely, we assume that
\begin{equation}\label{conditione-d1} \delta+\beta \le \frac{3d-2}{2d-2} \quad \mbox{and}\quad \frac{1}{2d-2}\le \beta -\varepsilon.\end{equation}
Then using the second part of Lemma \ref{lem4.3}, we have
 \begin{align}\label{L2good} \nonumber \|\widehat{g}\|_{L^2(P, d\sigma)} &\lesssim |G|^{\frac{1}{2}} +
 |G|^{\frac{11}{16}} |L_G|^{\frac{3}{16}} q^{\frac{-3d+5}{32}}+|G|^{\frac{5}{8}} |L_G|^{\frac{1}{4}} q^{\frac{-d+2}{16}} \\
 &\le |G|^{\frac{1}{2}} +|G|^{\frac{11}{16}} q^{\frac{-3d+11}{32}} + |G|^{\frac{5}{8}}q^{\frac{-d+6}{16}},
 \end{align}
 where we utilized the fact that $|L_G|\le q.$ As before, by comparing this estimate with Lemma \ref{lem3.3}, we select the $\varepsilon >0$ such that $|G|\le q^{\frac{3d+5}{6}}=  q^{\frac{d^2}{2d-2}+\varepsilon}.$
 Namely, we take $\varepsilon=\frac{2d-5}{6d-6}.$ With this $\varepsilon$,  if we choose $\frac{1}{3} \le \beta \le 1$ and $ \delta=\frac{d}{2d-2},$ then all conditions in \eqref{conditione-d1} hold, because $1\le |L_G|=q^\beta \le q.$
 \begin{remark}\label{rem1} In conclusion, we have seen that if $g$ is a regular function with its support $G\subset (\mathbb F_q^d, dm)$ such that  $q^{\frac{d^2}{2d-2}-\delta} \le |G| \le  q^{\frac{d^2}{2d-2}+\varepsilon}$
 and $ q^{\frac{1}{3}} \le |L_G|\le q$ for $\varepsilon=\frac{2d-5}{6d-6}$ and $\delta=\frac{d}{2d-2},$ then the inequality \eqref{L2good} holds. \end{remark}

 Now, we are ready to give the complete proof of Theorem \ref{main2-2}. \\

 \noindent{(\bf Case 1)} Assume that  $q^{\frac{d}{2}} \le |G| \le  q^{\frac{3d+5}{6}}$ which is the case in \eqref{size2} for $\varepsilon=\frac{2d-5}{6d-6}$ and $\delta=\frac{d}{2d-2}.$
 In addition, assume that $q^{\frac{1}{3}} \le |L_G|\le q.$ Then,  by Remark \ref{rem1} and the direct comparison, we see that if $q^{\frac{d}{2}} \le |G| \le  q^{\frac{3d+5}{6}}$ and $q^{\frac{1}{3}} \le |L_G|\le q,$ then
 for $d\ge 7,$
 \begin{align*}  \|\widehat{g}\|_{L^2(P, d\sigma)} &\lesssim |G|^{\frac{1}{2}} +|G|^{\frac{11}{16}} q^{\frac{-3d+11}{32}} + |G|^{\frac{5}{8}}q^{\frac{-d+6}{16}}\\
                                                   &\lesssim |G|^{\frac{3d+11}{6d+10}} =\|G\|_{L^{\frac{6d+10}{3d+11}}(\mathbb F_q^d, dm)} \sim \|g\|_{L^{\frac{6d+10}{3d+11}}(\mathbb F_q^d, dm)}. \end{align*}
 On the other hand, if $ 1\le |L_G|\le q^{\frac{1}{3}}$ and $ q^{\frac{d}{2}} \le |G| \le  q^{\frac{3d+5}{6}},$ then we see from the second part of Lemma \ref{key1} and the direct comparison that
 \begin{align*}  \|\widehat{g}\|_{L^2(P, d\sigma)} &\lesssim |G|^{\frac{d^2+d-1}{2d^2}}|L_G|^{\frac{1}{4}} \le |G|^{\frac{d^2+d-1}{2d^2}} q^{\frac{1}{12}} \lesssim |G|^{\frac{3d^2+4d-3}{6d^2}}\\
                                                    &= \|G\|_{L^{\frac{6d^2}{3d^2+4d-3}}(\mathbb F_q^d, dm)} \le \|G\|_{L^{\frac{6d+10}{3d+11}}(\mathbb F_q^d, dm)}
                                                    \sim \|g\|_{L^{\frac{6d+10}{3d+11}}(\mathbb F_q^d, dm)}.\end{align*}

Thus, Theorem \ref{main2-2} holds for all $q^{\frac{d}{2}} \le |G| \le  q^{\frac{3d+5}{6}}.$\\

\noindent{(\bf Case 2)} Assume that $1 \le |G| \le q^{\frac{d}{2}}.$ In this case, Theorem \ref{main2-2} can be proved by using Lemma \ref{lem3.5} and the direct comparison as follows:
$$\|\widehat{g}\|_{L^2(P, d\sigma)} \lesssim |G|^{\frac{1}{2}} + q^{\frac{-d+1}{4}} |G|\lesssim |G|^{\frac{3d+11}{6d+10}}=\|G\|_{L^{\frac{6d+10}{3d+11}}(\mathbb F_q^d, dm)}
                                                    \sim \|g\|_{L^{\frac{6d+10}{3d+11}}(\mathbb F_q^d, dm)}.$$

\noindent{(\bf Case 3)}  Assume that  $ q^{\frac{3d+5}{6}}\le |G| \le  q^{d}.$ In this case, the statement of Theorem \ref{main2-2} holds by Lemma \ref{lem3.3} and the direct comparison as follows:
$$\|\widehat{g}\|_{L^2(P,d\sigma)} \le q^{\frac{1}{2}} |G|^{\frac{1}{2}} \lesssim |G|^{\frac{3d+11}{6d+10}}
=\|G\|_{L^{\frac{6d+10}{3d+11}}(\mathbb F_q^d, dm)} \sim \|g\|_{L^{\frac{6d+10}{3d+11}}(\mathbb F_q^d, dm)}.$$
By Cases $1,2,$ and $3,$  the proof of Theorem \ref{main2-2} is complete.
\end{proof}


\begin{table}[ht]
\caption{Progress on the finite field extension problem for paraboloids in lower dimensions}
\begin{center}
\begin{tabular}{|c|c|c|}
    \hline
    Dimension $d,$  &  & \\
    Field $\mathbb F_q$ &  $R^*_P(p\to r)\lesssim 1$      & Authors \\
    \hline
     $d=2,~$ general $q$ & $p=2,~r=4 ~\mbox{(S-T)}$   &Mockenhaupt and Tao \cite{MT04} ~(solution)\\
    \hline
   $d=3,$  & $p=2,~r=4 ~\mbox{(S-T)}$  &Mockenhaupt and Tao \cite{MT04}~(sharp)\\
    $-1$ a square& $p=2.25, ~r=3.6$   &M. Lewko \cite{Le14}~(sharp)\\
                   &$p=\frac{18-5\varepsilon}{8-5\varepsilon},~r=3.6-\varepsilon$ &M. Lewko \cite{Le14}~(sharp)\\
                   & for some $\varepsilon>0$ & \\
              &    $p=3,~r=3$& (conjectured)\\
    \hline
    $d=3,$& $p=2,~r>3.6$  & Mockenhaupt and Tao \cite{MT04}\\
   $-1$ not a square  &  $p>1.6,~r=4$  & Mockenhaupt and Tao \cite{MT04}\\
     (prime $q$)& $p=2,~r=3.6$  & A. Lewko and M. Lewko \cite{LL10}\\
     & $p=1.6,~r=4$  & A. Lewko and M. Lewko \cite{LL10}(sharp)\\
& $p=2,~r>3.6-\frac{1}{1035}$  & M. Lewko \cite{LL13}  \\
 & $p=2,~r=3 $ &(conjectured)\\
   \hline
    $d=3,$& $p=2,~r=3.6-\varepsilon$  & M. Lewko \cite{LL13} \\
   $-1$ not a square  &  for some $\varepsilon>0$  & \\
    &   $p=2,~r=3 $ &(conjectured)\\
   \hline

\end{tabular}
\end{center}
\label{tab:multicol}
\end{table}

\begin{table}[ht]
\caption{Progress on the finite field extension problem for paraboloids in higher dimensions}
\begin{center}
\begin{tabular}{|c|c|c|}
     \hline
    Dimension $d,$  &  & \\
    Field $\mathbb F_q$ &  $R^*_P(p\to r)\lesssim 1$      & Authors \\
   \hline
        $d\ge 4$ even, & $p=2,~r=\frac{2d+2}{d-1}$~(S-T) & Mockenhaupt and Tao \cite{MT04}\\
       general $q$   & $p=2,~r>\frac{2d^2}{d^2-2d+2}$ & Iosevich and Koh \cite{IK09}\\
           & $p> \frac{4d}{3d-2},~  r=4$ & Iosevich and Koh \cite{IK09}\\
           & $p=2,~r=\frac{2d^2}{d^2-2d+2}$ & A. Lewko and M. Lewko \cite{LL10}\\
           & $p=\frac{4d}{3d-2},~  r=4$ & A. Lewko and M. Lewko \cite{LL10}~(sharp)\\
           & $p=2,~r>\frac{6d+8}{3d-2}$  & Theorem 1.4\\
           & $p= \frac{2d^2}{d^2-d+2},~r=\frac{2d}{d-1}$  & (conjectured) \\
           & $p=2,~r=\frac{2d+4}{d}$ & (conjectured best $r$ for $p=2$)\\
   \hline
   $d\ge 5$ odd, & $p=2,~r=\frac{2d+2}{d-1}$~(S-T) & Mockenhaupt and Tao \cite{MT04}~(sharp)\\
     $-1$ a square   & $p=\frac{2d+2}{d-1},~r=\frac{2d+2}{d-1}-\varepsilon_d$ & M. Lewko \cite{Le14}\\
       & for some $\varepsilon_d>0$ & \\
       & $p=\frac{2d}{d-1},~r=\frac{2d}{d-1}$& (conjectured)\\
     \hline

        $d=4\ell+1$ for $\ell\in \mathbb N,$ & $p=2,~r=\frac{2d+2}{d-1}$~(S-T) & Mockenhaupt and Tao \cite{MT04}~(sharp)\\
              $-1$ not a square     & &  \\
       & $p=\frac{2d}{d-1},~r=\frac{2d}{d-1}$& (conjectured)\\
     \hline

   $d=4\ell+3$ for $\ell\in \mathbb N,$ & $p=2,~r=\frac{2d+2}{d-1}$~(S-T) & Mockenhaupt and Tao \cite{MT04}\\
              $-1$ not a square & $p=2,~r>\frac{2d^2}{d^2-2d+2}$ & Iosevich and Koh \cite{IK09}\\
       & $p> \frac{4d}{3d-2},~  r=4$ & Iosevich and Koh \cite{IK09}\\
       & $p=2,~r=\frac{2d^2}{d^2-2d+2}$ & A. Lewko and M. Lewko \cite{LL10}\\
           & $p=\frac{4d}{3d-2},~  r=4$ & A. Lewko and M. Lewko \cite{LL10}\\
           & $p=2,~r>\frac{6d+10}{3d-1}$  & Theorem 1.5\\
            & $p=\frac{2d^2+2d }{d^2+3},~ r=\frac{2d}{d-1}$ & (conjectured)\\
            & $p=2,~r=\frac{2d+6}{d+1}$ & (conjectured best $r$ for $p=2$)\\
     \hline

\end{tabular}
\end{center}
\label{tab:multicol}
\end{table}

\newpage

\bibliographystyle{amsplain}

\end{document}